%% file: degree-bound.tex
\documentclass[12pt]{article}

\usepackage{amsmath}
\usepackage{amssymb}
\usepackage{amsthm}
\usepackage{amsfonts}
\usepackage{amscd}
\usepackage{graphicx}

\usepackage{longtable}
\usepackage[usenames,dvipsnames]{color}
\usepackage{hyperref}
\usepackage{makeidx}
\usepackage[numbers,sort&compress]{natbib}
\usepackage[nottoc]{tocbibind}
\usepackage{enumitem} 
\usepackage{titling}
\usepackage[parfill]{parskip}
\usepackage{fancyvrb} 
\usepackage{tocloft} 

\makeatletter
\def\thm@space@setup{%
  \thm@preskip=\parskip \thm@postskip=0pt
}
\makeatother

\hypersetup{
    colorlinks=true,
    breaklinks=true,
    filecolor=blue,
    urlcolor=Violet,
    linkcolor=blue,
    citecolor=Green,
    anchorcolor=blue,
    frenchlinks=true,
    pdfborder={0 0 0},
    pdfpagelayout=TwoPageRight,
    pdfdisplaydoctitle=true,
    bookmarksnumbered=true
}

\theoremstyle{plain}
\newtheorem{lemma}{Lemma}
\newtheorem{proposition}{Proposition}
\newtheorem{theorem}{Theorem}
\newtheorem{corollary}{Corollary}

\theoremstyle{definition}
\newtheorem{definition}{Definition}
\newtheorem{remark}{Remark}
\newtheorem{example}{Example}
\newtheorem{algorithm}{Algorithm}

\makeindex

\input{command}

\title{
A degree bound for families of rational curves on surfaces
}
\author{Niels Lubbes}
\date{\today}

\begin{document}

\maketitle

\begin{abstract}
We give an upper bound for the degree of rational curves
in a family that covers a given birationally ruled surface in projective space.
The upper bound is stated in terms of the degree, sectional genus
and arithmetic genus of the surface.
We introduce an algorithm for constructing examples
where the upper bound is tight.
As an application of our methods we improve an inequality on lattice polygons.
\end{abstract}


\tableofcontents

\section{Introduction}

A \Mdef{parametrization} of a rational surface $S\subset \MbbP^n$
is a birational map
\[
\Mdashfun{f}{\MbbC^2}{S\subset \MbbP^n}{(s,t)}{(f_0(s, t) : \ldots : f_n(s, t))}.
\]
The \Mdef{parametric degree} of $S$ is defined as the minimum of
the set of integers of the form $\max\Mset{\deg f_i}{0\leq i \leq n}$ 
for some birational map $\Mdashrow{f}{\MbbC^2}{S}$.

An upper bound for the parametric degree over an algebraically closed field of characteristic 0
is given in \citep[Theorem~9]{sch2} in terms of the sectional genus and degree of $S$.
In \citep[Theorem~20]{sch3} bounds for the parametric degree over perfect fields
are expressed in terms of the level and keel (\DEF{levelandkeel}).
The upper bound in \citep[Theorem~9]{sch2} can be interpreted as an upper bound on the level.
The analysis of \cite{sch3} applied to toric surfaces
led to new inequalities for invariants of lattice polygons in \cite{sch4}.
In \citep[Section~2.7]{cst1} it is conjectured that the inequality
can be improved by taking into account the number of vertices.
In \citep[Theorem~5]{sch5} these inequalities for lattice polygons are
translated to inequalities of rational surfaces
and in \citep[Section~4]{sch5}
the conjecture of \cite{cst1} is restated in the context of rational surfaces.

In this paper we generalize the level and keel for
rational surfaces in \citep[Section~3]{sch3} to birationally ruled surfaces
(this generalization is also posed as an open question in \citep[Section~1]{sch5}).
Instead of the parametric degree we now consider the minimal family degree (defined in \SEC{family}).
\THM{bound} gives an upper bound for the level of a birationally ruled surface
$S\subset\MbbP^n$
in terms of the sectional genus, degree and arithmetic genus.
As \COR{family} we obtain an upper bound for the minimal family degree.
If $S$ is rational, then our upper bound for the level coincides
with the upper bound for the level in \citep[Lemma~8]{sch2}.
However, in order to generalize this bound we give an alternative proof.
This proof enables us to make a case distinction on the invariants of $S$,
which improves the upper bound for the level.
Moreover, these methods enables us to prove the correctness
of \ALG{construct} that outputs examples where our upper bound is attained.
Thus we show that our upper bound for the level is tight in a combinatorial sense.
This algorithm is simple but has a non-trivial correctness proof.
These methods generalize the inequality \citep[Theorem~5]{sch5} to birationally ruled surfaces.
If we restrict our generalized
inequality to toric surfaces, we obtain an improved inequality
involving lattice polygons as conjectured in \citep[Section~2.7]{cst1},\citep[Section~4]{sch5}.
In light of the historical context,
one might ask whether this inequality
can be improved using the language of lattice geometry.

I would like to end this introduction with some additional remarks
on the degree of minimal parametrizations.
Let $s(f):=\max\Mset{\deg_sf_i}{0\leq i \leq n}$, $t(f):=\max\Mset{\deg_tf_i}{0\leq i \leq n}$
and we assume \Mwlog that $t(f)\geq s(f)$.
The \Mdef{parametric bi-degree} of $S$ is defined as
the minimum of $(s(f),t(f))$ among all birational maps  
$\Mdashrow{f}{\MbbC^2}{S}$
\Mwrt the lexicographic order
on ordered pairs of integers.
If $S\subset\MbbP^n$ attains at least two minimal families, then 
the parametric bi-degree of $S$ equals $(v(S),v(S))$
where $v(S)$ is minimal family degree \citep[Theorem~17]{nls-f3}.
Thus in this case our upper bound for the minimal family degree translates into
an upper bound of the parametric bi-degree.
If $S$ carries only one minimal family, then
an upper bound for the parametric bi-degree is still open.
In this case we also have to incorporate the keel in addition to
the sectional genus, degree and arithmetic genus of $S$.

\section{Intersection theory}
\label{sec:theory}

We recall some intersection theory and this section can be omitted by the expert.

The \Mdef{Neron-Severi group} $N(X)$ of a non-singular projective surface
$X$ can be defined as the group of divisors modulo numerical equivalence.
This group admits a bilinear intersection product
\[
\Mrow{\cdot}{N(X)\times N(X)}{\MbbZ}.
\]
The  \Mdef{Picard number} of $X$ is defined as the rank of $N(X)$.
The \Mdef{Neron-Severi theorem} states that the Picard number is finite.
For proofs in the next section we implicitly also consider $N(X)\otimes\MbbR$.
Moreover, we switch between the linear and numerical equivalence class of a divisor where needed.

The class of an \Mdef{exceptional curve} $E$ in $N(X)$ is characterized by
\[
E^2=E\cdot K=-1,
\]
where $K$ is the canonical divisor class of $X$.
\Mdef{Castelnuovo's contractibility criterion} states that for all
exceptional curves $E$ there exists a contraction map
\[
\Mrow{f}{X}{Y},
\]
\Mst $f(E)=p$ with $p$ a smooth point and
$
\Marrow{(X\setminus E)}{}{(Y\setminus p)}
$
is an isomorphism via $f$.
The assignment of Neron-Severi groups is functorial \Mst
\[
\Mrow{f^*}{N(Y)}{N(X)}.
\]
The groups are related by
\[
N(X)\cong N(Y)\oplus\MbbZ\Mmod{E},
\]
and thus the Picard number drops for each contracted curve.
The formula for pullback of the canonical class is
\[
f^*(K)=K_Y-E.
\]
Suppose that $D\subset Y$ is a divisor and let $\tilde{D}$ be the strict transform of $D$ along 
$\Mrow{f}{X}{Y}$. In this case
\[
f^*[D]=[\tilde{D}]+mE,
\]
where $[D]\in N(Y)$, $[\tilde{D}]\in N(X)$
and $m$ is the order of $D$ at $p$.
For the intersection product we
have the \Mdef{projection formula}
\[
f^*(C)\cdot A=C\cdot f_*(A),
\]
and compatibility with the pullback
\[
f^*(A)\cdot f^*(B)=A\cdot B,
\]
for all $A,B \in N(X)$ and $C\in N(Y)$.

The \Mdef{Hodge index theorem} states that if $A^2>0$ and $A\cdot B=0$, then $B^2<0$ or $B=0$ for all $A,B\in N(X)$.

The \Mdef{adjunction formula} implies that $A^2+A\cdot K\geq -2$ for all $A\in N(X)$.
If $D$ is a divisor isomorphic to $\MbbP^1$, then 
$[D]^2+[D]\cdot K = -2$ with $[D]\in N(X)$.

We denote by $p_a(X)$ the \Mdef{arithmetic genus} of $X$ and it is a birational invariant.
If $X$ is a ruled surface then $p_a(X)$ equals the negative of the geometric genus of its base curve.

The \Mdef{Riemann-Roch theorem} states that
\[
h^0(D)-h^1(D)+h^2(D)=\frac{D\cdot (D-K)}{2}+p_a(X)+1,
\]
for a divisor class $D$ (up to linear equivalence) with associated sheaf $\McalO(D)$.
Here $h^i(D)$ denotes the dimension of the i-th sheaf cohomology $\dim H^i(X, \McalO(D) )$.
\Mdef{Serre duality} states that $h^2(D)=h^0(K-D)$.

\section{Adjunction}
\label{sec:adjunction}

Adjunction works over any field.

We call a divisor class $D$ of a surface \Mdef{efficient} \Miff $D\cdot E>0$ for all exceptional curves $E$.

We define a \Mdef{ruled pair} as a pair $(X,D)$
where $X$ is a non-singular birationally ruled surface and $D$ is a nef and efficient divisor class of $X$.

If $D$ is effective, then the \MdefAttr{polarized model}{ruled pair} of $(X,D)$ is defined as
$\overline{\varphi_D(X)} \subset \MbbP^{h^0(D)-1}$ where $\varphi_D$ is the map associated
to the global sections $H^0(X,\McalO(D))$.

If $(X,D)$ is a ruled pair, then 
$X$ has Kodaira dimension $-\infty$
and thus the canonical divisor class $K$ of $X$ is not nef \citep[Section~1.2]{mat1}.
The \Mdef{nef threshold} of $D$ is defined as
\[
t(D):=\textrm{sup}\Mset{~q\in\MbbR~}{~ D+qK \text{ is nef}~}.
\]

We call a ruled pair $(X,D)$ \MdefAttr{non-minimal}{ruled pair}
if $D$ is big and either
\begin{enumerate}[topsep=0pt, label=(\roman*)]
\item $t(D)=1$ and $D\neq -K$, or 
\item $t(D)>1$.
\end{enumerate}
We call a ruled pair $(X,D)$ \MdefAttr{minimal}{ruled pair}
if either
\begin{enumerate}[topsep=0pt, label=(\roman*)]
\item $t(D)=1$ and $D= -K$, or
\item $t(D)<1$.
\end{enumerate}
An \Mdef{adjoint relation} is a relation between two ruled pairs
\[
\Marrow{(X,D)}{\mu}{(X',D')}:=\bigl(\mu(X),\mu_*(D+K)\bigr),
\]
\Mst
$(X,D)$ is a non-minimal ruled pair and
$\Mrow{\mu}{X}{X'}$ is a
birational morphism that contracts all exceptional curves $E$ \Mst $(D+K)\cdot E=0$.

\begin{lemma}
\textbf{(adjoint relation)}
\label{lem:D}
\\
Let $\Marrow{(X,D)}{\mu}{(X',D')}$ be an adjoint relation.
\begin{itemize}[topsep=0pt]

\Mmclaim{a} $\mu^*D'=D+K$ and $D'^2=(D+K)^2$.

\Mmclaim{b} If $D'^2>0$, then $\Mrow{\mu}{X}{X'}$ is unique
up to biregular isomorphism.

\end{itemize}
\end{lemma}

\begin{proof}
Let $(E_j)_j$ be the curves that are contracted by $\Mrow{\mu}{X}{X'}$.

\Mrefmclaim{a}
See \SEC{theory} for the pullback of a divisor class along a contraction map
and the compatibility of pullback with the intersection product.
From $(D+K)\cdot E_j=0$ it follows that $\mu^*D'=D+K$ and thus $D'^2=(D+K)^2$.

\Mrefmclaim{b}
From the Hodge index theorem, $(D+K)^2>0$ and $(D+K)\cdot(E_1+E_2)=0$ it follows that $(E_1+E_2)^2<0$
and thus $E_1\cdot E_2=0$.
It follows that if $D'^2>0$, then the contracted exceptional curves are disjoint.
The contraction of an exceptional curve is an isomorphism outside this exceptional curve.
Thus the order of contracting disjoint curves does not matter up to biregular isomorphism.
\end{proof}

\begin{proposition}
\textbf{(adjoint relation)}
\label{prp:relation}
\\
If $\Marrow{(X,D)}{\mu}{(X',D')}$ is an adjoint relation, then 
$(X',D')$ is either a non-minimal or a minimal ruled pair.

\end{proposition}

\begin{proof}
We use the pullback formulas for divisor classes,
its compatibility with the intersection product and
the projection formula as described in \SEC{theory}.

Suppose by contradiction that $D'$ is not nef.
It follows that there exists a curve $C'$ \Mst $D'\cdot C'=\mu^*D'\cdot \mu^*C'<0$.
From $\mu^*D'\cdot \mu^*C'<0$ and \LEM{D}.\Mrefmclaim{a}
it follows that $(D+K)\cdot C<0$, where $C$ is the strict transform of $C'$.
However, the nef threshold $t(D)$ is greater or equal to one.
We have thus arrived at a contradiction.

Suppose by contradiction that $D'$ is not efficient.
From \LEM{D}.\Mrefmclaim{a} it follows that there exists exceptional curve $E'$
\Mst $D'\cdot E'=\mu^*D'\cdot \mu^*E'=(D+K)\cdot E=0$ where $E$ is the strict transform of $E'$.
We find that $K'\cdot \mu_*E=\mu^*K'\cdot E=-1$ and thus $K\cdot E\leq -1$.
From
\[
\mu^*K'\cdot \mu^*E'=\left(K-\Msum{j}{}E_j\right)\cdot \left(E+\Msum{j}{}m_jE_j\right)=K\cdot E - \Msum{a\neq b}{}m_aE_a\cdot E_b=-1,
\]
it follows that $K\cdot E\geq -1$.
From the adjunction formula and $E\cong\MbbP^1$ it follows that $E^2+E\cdot K=-2$.
Therefore $E^2=E\cdot K=-1$ and thus $E$ is an exceptional curve not contracted by $\mu$.
We arrived at a contradiction.
\end{proof}

We call a minimal ruled pair $(X,D)$ a \MdefAttr{weak Del Pezzo pair}{minimal ruled pair} \Miff
either
$D=-K$, $D=-\frac{1}{2}K$, $D=-\frac{1}{3}K$, or $D=-\frac{2}{3}K$, with $K$ the canonical divisor class of $X$.

We call a minimal ruled pair $(X,D)$ a
\MdefAttr{geometrically ruled pair}{minimal ruled pair}
\Miff $\Mrow{\varphi_{M}}{X}{C}$ is a geometrically ruled surface
\Mst either
$M=aD$, or $M=a(2D+K)$
for large enough $a\in\MbbZ_{>0}$.
Here
$\varphi_M$ is the map associated to the global sections $H^0(X,\McalO(M))$,
$C=\varphi_{M}(X)$ and $K$ is the canonical divisor class of $X$.

\begin{proposition}
\textbf{(Neron-Severi group of a minimal ruled pair)}
\label{prp:neron}
\\
Let $(X,D)$ be a minimal ruled pair, with $K$ be the canonical divisor class of $X$
and $N(X)$ the Neron-Severi group. Let $p$ denote the arithmetic genus of $X$.
\begin{itemize}[topsep=0pt]
\Mmclaim{a}
If $(X,D)$ is a weak Del Pezzo pair with $K^2\neq 8$, then $p=0$,
\\
$
N(X)\cong\MbbZ\Mmod{ H, Q_1,\ldots, Q_r }
$
with $0\leq r=9-K^2 \leq 8$ and
intersection product $H\cdot Q_i=0$, $Q_i^2=1$ and $Q_i\cdot Q_j=0$ for $i\neq j$ in $[1,r]$.
We have $-K=3H-Q_1-\ldots - Q_r$ and either $D=-K$, $D=-\frac{1}{3}K$, or $D=-\frac{2}{3}K$.

\Mmclaim{b}
If $(X,D)$ is a weak Del Pezzo pair with $K^2=8$, then $p=0$,
$
N(X)\cong\MbbZ\Mmod{ H, F }
$
with intersection product $H^2=r$, $H\cdot F=1$ and $F^2=0$ for $r\in\{0,1,2\}$.
We have $K=-2H+(r-2)F$ and either $D=-K$ or $D=-\frac{1}{2}K$.

\Mmclaim{c}
If $(X,D)$ is a geometrically ruled pair, then
$
N(X)\cong\MbbZ\Mmod{ H, F }
$
with intersection product $H^2=r$, $H\cdot F=1$ and $F^2=0$ for $r\in\MbbZ_{\geq0}$.
Either $D=kF$ or $2D+K=kF$ for $k\in\MbbZ_{>0}$ and $K=-2H+(r-2p-2)F$
\Mst $K^2=8(p+1)$.
\end{itemize}

\end{proposition}

\begin{proof}
For \Mrefmclaim{a} and \Mrefmclaim{b} see \citep[Section~8.4.3]{dol1}.
For \Mrefmclaim{c} see \citep[Chapter~3, Proposition~18, page~34]{bea1}. 
\end{proof}

An \Mdef{adjoint chain} of a ruled pair is defined as a chain of
successive adjoint relations until a minimal ruled pair is obtained.

\newpage
\begin{proposition}
\textbf{(adjoint chain)}
\label{prp:chain}
\\
An adjoint chain must terminate and a minimal ruled pair
at the end is either a weak Del Pezzo pair or a geometrically ruled pair.
\end{proposition}

\begin{proof}
Let $(X,D)$ be a non-minimal ruled pair and let $t:=t(D)$ be the nef threshold.
Let $K$ be the canonical class of $X$.
From \citep[Corollary~1-2-15]{mat1} it follows that $t\in\MbbQ_{>0}$ with denominator bounded by $3$.
After a $\lfloor t \rfloor$ successive adjoint relations 
$D$ is replaced by the pushforward of $D+\lfloor t \rfloor K$
and thus we may assume that $t\leq 1$.
We make a case distinction.

First suppose that $(D+tK)^2>0$.
There exists an irreducible curve $C$ \Mst $(D+tK)\cdot C=0$, $D\cdot C>0$ and $K\cdot C<0$.
From the Hodge index theorem and $(D+tK)^2>0$ it follows that $C^2<0$.
From the adjunction formula it follows that $C^2+K\cdot C=-2$.
From \citep[Lemma~1-1-4]{mat1} it follows that $C$ is an exceptional curve.
Recall from \SEC{theory} that the Picard number drops for each contracted exceptional curve
and that this number is finite.

Next, if $(D+tK)^2=0$ and $D= -tK$, then $(X,D)$ is a weak Del Pezzo pair.

Finally, we assume that $(D+tK)^2=0$ and $D\neq -tK$.
If $t=1$, then we apply one extra adjoint relation so we may assume that $t<1$.
From \citep[Theorem~1-2-14 and Proposition~1-2-16]{mat1}
it follows that
that the map associated to $a(D+tK)$, with large enough $a\in\MbbZ_{>0}$, 
defines a Mori fibre space \citep[Definition~1-4-1]{mat1}.
It follows from \citep[Theorem~1-4-4]{mat1} that a fibre $F$ 
of this morphism is isomorphic to $\MbbP^1$ with $F^2=0$.
By the adjunction formula we have $F\cdot K=-2$.
From $(D+tK)\cdot F=0$ it follows that $t=\frac{D\cdot F}{2}\in\frac{1}{2}\MbbZ_{\geq 0}$.
Thus in this case $(X,D)$ is a geometrically ruled pair.
\end{proof}

\begin{definition}
\textbf{(level and keel)}
\label{def:levelandkeel}
\\
Suppose
$
\Marrow{(X_0,D_0)}{\mu_0}{(X_1,D_1)}\Marrow{}{\mu_1}{}
\ldots
\Marrow{}{\mu_{\ell-1}}{(X_\ell,D_\ell)}
$
is an adjoint chain.

The \MdefAttr{level}{ruled pair} of $(X_0,D_0)$ is defined by $\ell$.
The \MdefAttr{keel}{ruled pair} of $(X_0,D_0)$ is either:
\begin{itemize}[topsep=0pt]
\item $0$ if $(X_\ell,D_\ell)$ is a weak Del Pezzo pair, or
\item $k$ as in \PRP{neron}.\Mrefmclaim{c} if $(X_\ell,D_\ell)$ is a geometrically ruled pair.
\Mend
\end{itemize}
\end{definition}

\begin{proposition}
\textbf{(level and keel)}
\label{prp:levelkeel}
\\
The level and keel are well defined.

\end{proposition}

\begin{proof}
Let $\Marrow{(X,D)}{\mu}{(X',D')}$ be an adjoint relation.

Since $D'^2=0$ can only occur at the last adjoint relation in an adjoint chain,
it follows from \LEM{D}.\Mrefmclaim{b} that the level is well defined.

We now show that also the keel does not depend on the last adjoint relation
and thus is uniquely defined.
Suppose that $(X',D')$ is a geometrically ruled pair.
From \PRP{neron} and \LEM{D}.\Mrefmclaim{a} it follows that if $D'^2=(D+K)^2=0$, then 
$-2k=\mu^*D'\cdot \mu^*K'=(D+K)\cdot K$ defines the keel $k$.
Similarly, if $D'^2=(D+K)^2>0$, then 
$-2k=\mu^*(2D'+K')\cdot \mu^*K'=2(D+K)\cdot K+K'^2$.
From \PRP{neron} it follows that $K'^2=8(p+1)$, 
where the arithmetic genus $p$ is a birational invariant.
Thus our assertion holds.
\end{proof}

\begin{remark}
\label{rmk:levelkeel}
\textbf{(level and keel)}
\\
The level and keel have been introduced in \citep[Section~3]{sch3}.
In our generalization to birationally ruled surfaces
we use a slightly alternative definition for the level,
since it simplifies our arguments.
If
$D_\ell=-K_\ell$,
$D_\ell=-\frac{1}{2}K_\ell$,
$D_\ell=-\frac{1}{3}K_\ell$,
$D_\ell=-\frac{2}{3}K_\ell$,
$D_\ell=kF$ or
$2D_\ell+K_\ell=kF$
as in \PRP{neron}, then we define $\lambda$
as
$1$,
$\frac{1}{2}$,
$\frac{1}{3}$,
$\frac{2}{3}$,
$0$ or
$\frac{1}{2}$ respectively.
Now the level in \citep[Section~3]{sch3} for rational surfaces is defined as $\ell+\lambda$.
\Mend
\end{remark}

\section{Minimal families}
\label{sec:family}

A \Mdef{family of curves} $F$ for ruled pair $(X,D)$ that is indexed by a smooth curve $C$,
is defined as a divisor $F\subset X\times C$ \Mst the first projection $\Marrow{F}{}{X}$
is dominant.
If the generic curve of $F$ is rational and if $D\cdot F$ is minimal \Mwrt all
families of rational curves, then we call $F$ \MdefAttr{minimal}{family of curves}.
The \Mdef{minimal family degree} $v(X,D)$ is defined as $D\cdot F$ for a minimal family $F$.
Note that since $(X,D)$ is a ruled pair, there always exists a minimal family.

We recall part of \citep[Theorem~46]{nls1} concerning the degree of minimal families
along an adjoint relation $\Marrow{(X,D)}{\mu}{(X',D')}$.
If $X\cong X'\cong\MbbP^2$, then
$v(X,D) = v(X',D') + 3$,
else
$v(X,D) = v(X',D') + 2$.
If $(X,D)$ is a weak Del Pezzo pair and $X\cong\MbbP^2$, then $v(X,D)\leq 3$.
If $(X,D)$ is a weak Del Pezzo pair and $D^2=8$, then $v(X,D)\leq 2$.
If $(X,D)$ is a weak Del Pezzo pair and $D^2<8$, then $v(X,D)= 2$.
If $(X,D)$ is a geometrically ruled pair, then $v(X,D)\leq 1$.

\section{Upper bound for the level}

\subsection{}
\label{sec:bound}

Let
$
\Marrow{(X_0,D_0)}{\mu_0}{(X_1,D_1)}\Marrow{}{\mu_1}{}
\ldots
\Marrow{}{\mu_{\ell-1}}{(X_\ell,D_\ell)}
$
be an adjoint chain.
From now on let $K_i$ denote the canonical class of $X_i$.
We introduce the following notation:
\[
\alpha(i)=D_i^2,\quad
\beta(i)=D_i\cdot K_i,\quad
\gamma(i)=K_i^2, \quad
h(i)=D_i^2-D_i\cdot K_i,
\]
and $n(i)$ denotes the number of curves contracted by $\mu_i$ for $0 \leq i \leq \ell$.

\begin{lemma}
\textbf{(adjoint intersection products)}
\label{lem:step}
\\
If $\ell>0$, then
\begin{itemize}[topsep=0pt]
\Mmclaim{a} $\alpha(i+1)=\alpha(i)+2\beta(i)+\gamma(i)$,

\Mmclaim{b} $\beta(i+1)=\beta(i) + \gamma(i)$,

\Mmclaim{c} $\gamma(i+1)=\gamma(i) + n(i)$,

\Mmclaim{d} $h(i+1)=h(i) + 2\beta(i)$,
\end{itemize}
for $0 \leq i \leq \ell-1$.

\end{lemma}

\begin{proof}
We use the pullback formulas for divisor classes,
its compatibility with the intersection product and
the projection formula as described in \SEC{theory}.
Now \Mrefmclaim{a} and \Mrefmclaim{b} are a straightforward consequence of \LEM{D}.\Mrefmclaim{a}.
Let $(E_j)_j$ be the curves that are contracted by $\Mrow{\mu_i}{X_i}{X_{i+1}}$.
For \Mrefmclaim{c} we compute
\[
K_i^2
=
(\mu_i^*K_{i+1})^2
=
K_{i+1}^2 - n(i) + \Msum{j\neq k}{}  E_j\cdot E_k,
\]
and we need to show that $\Msum{j\neq k}{}  E_j\cdot E_k=0$.
From $\mu_{i*}E_k=0$ it follows that $K_{i+1}\cdot \mu_{i*}E_k=0$ and thus
\[
K_{i+1}\cdot \mu_{i*}E_k=\mu_i^*K_{i+1}\cdot E_k=\left(K_i-\Msum{j}{}E_j\right)\cdot E_k=-1+1 -\Msum{j\neq k}{}E_j\cdot E_k=0.
\]
This proves \Mrefmclaim{c}.
From $h(i+1)=\alpha(i+1)-\beta(i+1)=\alpha(i) + \beta(i) = h(i) +2\beta(i)$
it follows that \Mrefmclaim{d} holds.
\end{proof}

\begin{remark}
\LEM{step}.\Mrefmclaim{d} is essentially \citep[Lemma~7]{sch2},
which is attributed there to Castelnuovo \citep[Remark~3]{sch2}.
\Mend
\end{remark}

We say that $(X_i,D_i)$
has \Mdef{adjoint state} $S_a(i)$ for some $1 \leq a \leq 4$ 
if $\gamma(i)$ and $\beta(i)$ are as in the following table:
\begin{center}
\begin{tabular}{|c|c|c|}
\hline
adjoint state  & $\gamma(i)$ & $\beta(i)$    \\\hline
\hline
$S_1(i)$ & $<0$        & $\geq 0$            \\\hline
$S_2(i)$ & $<0$        & $<0$                \\\hline
$S_3(i)$ & $=0$        & $<0$                \\\hline
$S_4(i)$ & $>0$        & $<0$                \\\hline
\end{tabular}
\end{center}
where $0\leq i \leq \ell$.

\begin{lemma}
\textbf{(adjoint states)}
\label{lem:state}
\begin{itemize}[topsep=0pt]

\Mmclaim{a}
A ruled pair $(X_i,D_i)$ 
has adjoint state either
$S_1(i)$, $S_2(i)$, $S_3(i)$ or $S_4(i)$ for all $0\leq i \leq \ell$.

\Mmclaim{b}
If $S_a(i)$ and $S_b(i+1)$ for $0 \leq i \leq \ell-1$, then $a\leq b$.

\end{itemize}
\end{lemma}

\begin{proof}
~\\
\Mrefmclaim{a}
We assume first that $\gamma(i)=0$.
Assume by contradiction that $\beta(i)\geq 0$.
From \LEM{step} it follows that $\alpha(j+1)\geq \alpha(j)$ and $\beta(j)=\beta(j+1)$ for all $j\geq i$.
But then the adjoint chain is of infinite length.
We have thus arrived at a contradiction.

Next we assume that $\gamma(i)>0$.
From \PRP{neron} it follows that $p=\min(0, \Mceil{\frac{1}{8}\gamma(\ell)-1})$ and thus $p=0$.
From the Riemann Roch theorem and Serre duality it follows that $h^0(-K_i)\geq \gamma(i)+1>0$.
From $D_i$ being nef it follows that $\beta(i)\leq 0$.
From the Hodge index theorem it follows that $\beta(i)<0$.

\Mrefmclaim{b}
From \LEM{step} it follows that $\gamma(i)<\gamma(i+1)$ and if $\gamma(i)<0$, then $\beta(i+1)<\beta(i)$.
\end{proof}

\begin{lemma}
\textbf{(dimension)}
\label{lem:dim}
\\
We have that $h(i)\geq 2$ for $1\leq i\leq \ell$.

\end{lemma}

\begin{proof}
By \LEM{step} we have $h(0)+h(1)=2\alpha(0)>0$.
Thus $h(0)>0$ and the base case $h(1)>0$ of the induction holds.
By induction hypothesis $h(i)>0$.
The induction step is to show that $h(i+1)>0$.
If $\beta(i)\geq 0$, then from \LEM{step}  it follows that $h(i+1)=h(i)+2\beta(i)>0$.
If $\beta(i)<0    $, then by \LEM{state} we have $\beta(i+1)<0$ and thus
$h(i+1)=\alpha(i+1)-\beta(i+1)>0$.
We can conclude from the Riemann-Roch theorem that $h(i)$ must be even and thus $h(i)\neq 1$.
\end{proof}

\subsection{}
\label{sec:bound2}

We will now consider the following combinatorial problem.
For given
\[
h(0), \beta(0), p \in \MbbZ,
\]
find an upper bound for $\ell$ such that there exists a sequence of integer 3-tuples
\[
( h(i), \beta(i), \gamma(i))_{0 \leq i \leq \ell},
\]
which adheres to the following 5 rules:
\begin{itemize}[topsep=0pt]
\item[(H)] $h(i+1)=h(i)+2\beta(i)$ (\LEM{step}.\Mrefmclaim{d}).
\item[(B)] $\beta(i+1)=\beta(i)+\gamma(i)$ (\LEM{step}.\Mrefmclaim{b}).
\item[(Z)] $h(i)\geq 2$ for $1\leq i \leq l$ (\LEM{dim}).
\item[(S)] $S_a(i)$ for $1\leq a \leq 4$ and if $S_a(i)$ and $S_b(i+1)$, then $a\leq b$ (\LEM{state}).
\item[(P)] $p\leq 0$. If $p=0$, then $\gamma(\ell)>0$. If $p<0$, then $\gamma(\ell)=8(p+1)$. (\PRP{neron}).
\end{itemize}
Using \PRP{neron} it is possible to impose restrictions on $(\gamma(\ell),\beta(\ell))$ and
from the Riemann-Roch theorem we can conclude that $h(i)$ must be even.
However, we do not need these additional rules for a proof.
For our solution of the posed problem we
make a case distinction between $S_1(0)$, $S_2(0)$, $S_3(0)$ and $S_4(0)$.
In the proof of \THM{bound} we will compose the upper bounds of these
cases.

First we start with a technical lemma for convenience.

\begin{lemma}
\textbf{(technical lemma)}
\label{lem:hj}
\\
If $\gamma(0)=\ldots=\gamma(j-1)$ for some $1\leq j\leq \ell$, then
\[
h(j)=\gamma(0)j^2 + (2\beta(0)-\gamma(0))j + h(0).
\]
\end{lemma}

\begin{proof}
With (H) we expand $h(j)$ \Mst
\[
h(j)
=
h(j-1) + 2\beta(j-1)
=
\ldots
=
h(0) + 2\Msum{n=0}{j-1}\beta(n).
\]
With (B) we expand the $\beta(i)$ terms \Mst
\[
h(j)= h(0)+2\Msum{n=0}{j-1}(\beta(0)+n\gamma(0))=h(0)+2j\beta(0)+j(j-1)\gamma(0).
\]
We conclude this proof by re-arranging terms.
\end{proof}

\begin{lemma}
\textbf{(case $S_4(0)$)}
\label{lem:S4}
\\
If $S_4(0)$, then
\[
\ell\leq -\beta(0)-1 < \frac{h(0)-2}{2}.
\]
\end{lemma}

\begin{proof}
From (S) and (B) it follows that $\beta(0)<\ldots<\beta(\ell)<0$
and thus we conclude the first inequality.
The second inequality follows from $h(0)+2\beta(0)\geq 2$.
\end{proof}

\begin{lemma}
\textbf{(case $S_3(0)$)}
\label{lem:S3}

If $S_3(0)$, then
\[
\ell\leq \frac{h(0)-2}{2}.
\]
Moreover, if $\ell$ is equal to this upper bound, then $S_3(\ell-1)$.
\end{lemma}

\begin{proof}
It follows from (H) and (S) that $h(i+1)-h(i)=2\beta(i)\leq -2$
if $S_3(i)$ or $S_4(i)$ for all $0\leq i \leq \ell$.
The upper bound asserted in the lemma now follows from (Z).
This upper bound is attained if $\beta(i)=-1$ for $0\leq i \leq \ell$.
It follows that $S_4(i)$ \Miff $i=\ell$ and $p=0$.
Thus we can conclude from (S) that $S_3(\ell-1)$ in case of equality.
\end{proof}

For an example where the upper bound of \LEM{S3} is attained,
see $3\leq i \leq 7$ in Table 1 of \EXM{bound}.

\begin{lemma}
\textbf{(case $S_2(0)$ with $p\geq -1$)}
\label{lem:S2a}
\\
If $S_2(0)$ and $p\geq -1$, then
\[
\ell
<
\frac{h(0)-2}{2}.
\]
If $\beta(0) - \beta(\ell)>0$, then
\[
\ell
\leq
s + \Mbigfloor{\frac{-s^2 + (2\beta(0)+1)s + h(0) - 2}{-2\beta(\ell)}},
\]
where $s:=\beta(0) - \beta(\ell)$.
\end{lemma}

\begin{proof}
Suppose that $s>0$.
It follows from (S) and (B) that if
$S_3(k)$ or $S_4(k)$, then $\beta(k)\leq \beta(\ell)$ for $0\leq k\leq \ell$.
From (B) it follows that $k\geq s$
where we have equality if $\gamma(0)=\ldots=\gamma(k-1)=-1$.
It now follows from \LEM{hj} that
\[
h(k) \leq h(s)=-s^2 + (2\beta(0)+1)s  +h(0).
\]
It follows from (S) that $\beta(i)<0$ for $k\leq i \leq \ell$.
From (Z) and thus the same argument in \LEM{S3} it follows that
\[
\ell \leq s + \frac{h(s)-2}{-2\beta(\ell)}.
\]
The first inequality follows if $\beta(0)=\beta(\ell)=-1$ \Mst $s=0$.
\end{proof}

For an example where the second upper bound of \LEM{S2a} is attained,
see $6\leq i \leq 12$ in Table 2 of \EXM{bound}.

\begin{lemma}
\textbf{(case $S_2(0)$ with $p\leq -2$)}
\label{lem:S2b}
\\
If $S_2(0)$ with $p\leq -2$, then
\[
\ell\leq \Mbigfloor{\frac{-(2\beta(0)-t) - \sqrt{\Delta}}{ 2t }},
\]
where $\Delta=(2\beta(0)-t)^2-4t(h(0)-2)$ and $t:=8(p+1)$.
\end{lemma}

\begin{proof}
From (P) and (S) it follows that $\gamma(\ell)=8(p+1)$ and $S_2(\ell)$.
From (Z) and (H) it follows that $h(\ell)+2\beta(\ell)\leq 0$.
It follows from (H) that $h(i)$
decreases as slow as possible if $\gamma(0)=\ldots=\gamma(\ell)$.
It follows from (Z) that $h(\ell)\geq 2$
so that we can equate the formula of \LEM{hj} to $2$.
The upper bound for $\ell$ now follows from the quadratic formula.
\end{proof}

For an example where the upper bound of \LEM{S2b} is attained,
see $5\leq i \leq 9$ in Table 4 of \EXM{bound}.

\begin{lemma}
\textbf{(case $S_1(0)$)}
\label{lem:S1}
\\
If $S_1(0)$ and $j$ is the largest index \Mst $S_1(j-1)$, then
\[
j\leq \Mbigfloor{\frac{\beta(0)}{-t}}+1,
\]
and
\[
h(j)
\leq
t\left(\Mbigfloor{\frac{\beta(0)}{-t}}+1\right)^2
+
(2\beta(0)-t)\left(\Mbigfloor{\frac{\beta(0)}{-t}}+1\right) + h(0),
\]
where
\[
t:=\min(8(p+1),-1).
\]
Moreover, if the upper bound for $j$ and $h(j)$ is reached, then $\beta(j)=t$.

In case $p\geq -1$ then the upper bound for $h(j)$ simplifies to
\[
h(j)\leq \beta(0)^2+\alpha(0).
\]
\end{lemma}

\begin{proof}
From (S) and (P) we find that $\gamma(i)\leq t$ for $i<j$.
It follows from (S) and (B) that in order to find an upper bound for $j$
we need to assume that $\gamma(0)=\ldots=\gamma(j-1)=t$
\Mst
\[
\beta(j-1)= \beta(j-2)-\gamma(0) = \ldots = \beta(0)-(j-1)\gamma(0)=0.
\]
From this we conclude the upper bound for $j$.
The upper bound
for $h(j)$ is reached if we substitute the upper bound for $j$
in the formula of \LEM{hj}.
From (B) it follows that $\beta(j)=t$ if the upper bounds for $j$ and $h(j)$ are reached.
If $p\geq -1$, then $t$ divides $\beta(0)$ so that this formula simplifies.
\end{proof}

For an example where upper bound of \LEM{S1} is attained,
see $0\leq i \leq 2$ in Table 1 of \EXM{bound}.

\begin{theorem}
\textbf{(upper bound level)}
\label{thm:bound}
\\
We state upper bounds for the level in terms of $\alpha(0)$, $\beta(0)$ and $p$,
where $p$ is the arithmetic genus of $X_0$.

If $p=0$ or $p=-1$, then
\[
\ell \leq \frac{\beta(0)^2+\alpha(0)}{2} + \beta(0),
\]
and if moreover $\beta(0)<0$, then
\[
\ell \leq \frac{\alpha(0)-\beta(0)-2}{2}.
\]

If $p\leq -2$, then
\[
\ell\leq
\Mbigfloor{\frac{\beta(0)}{-t}}
+
1
+
\Mbigfloor{\frac{-t-\sqrt{t^2 - 4t(\Upsilon-2)}}{ 2t }},
\]
and if moreover $\beta(0)<0$, then
\[
\ell\leq \Mbigfloor{\frac{-(2\beta(0)-t) - \sqrt{\Delta}}{ 2t }},
\]
where
\[
t:=8(p+1)
\quad\Mand\quad
\Delta:=(2\beta(0)-t)^2 - 4t(\alpha(0) - \beta(0)-2),
\]
and
\[
\Upsilon:=
t\left(\Mbigfloor{\frac{\beta(0)}{-t}}+1\right)^2
+
(2\beta(0)-t)\left(\Mbigfloor{\frac{\beta(0)}{-t}}+1\right) + \alpha(0) - \beta(0).
\]
\end{theorem}

\begin{proof}
Recall that by definition $h(0)=\alpha(0)-\beta(0)$.
Thus in order to proof this theorem we
need to solve the problem as posed at the beginning of \SEC{bound2}.

First we assume that $-1 \leq p \leq 0$.
It follows from (S) that an upper bound of $\ell$ is obtained as
the composition of the upper bound of \LEM{S1}
with the upper bound of either \LEM{S2a}, \LEM{S3} or \LEM{S4}.
It follows that the upper bound of \LEM{S3} is the choice
which acquires the highest upper bound.
If $\beta(0)<0$, then we can apply the upper bound of \LEM{S3} directly.

If $p\leq -2$, then it follows from (S) and (P) that
the upper bound as asserted in this theorem can be obtained by composing the
upper bound of \LEM{S1} with the upper bound in \LEM{S2b}.
If $\beta(0)<0$, then we can apply the upper bound of \LEM{S2b} directly.
\end{proof}

\begin{corollary}
\textbf{(upper bound for the minimal family degree)}
\label{cor:family}
\\
Let $v=v(X_0,D_0)$ be the minimal family degree.
Let $\tilde{\ell}$ be the upper bound for the level from \THM{bound}.

If $p=0$, then $v \leq 2\tilde{\ell}+2$.

If $p\leq -1$, then $v \leq 2\tilde{\ell}+1$.
\end{corollary}

\begin{proof}
We recall from \SEC{family} that
if $X_i\cong X_{i+1}\cong\MbbP^2$, then
$p=0$ and
\[
v(i+1)=v(i)+3,
\]
where $v(i):=v(X_i,D_i)$.
Otherwise $v(i+1)=v(i)+2$.
We want to show that $3\ell+3< 2\tilde{\ell}+2$
and thus we may assume that the minimal family degree is
increased by $2$ at each step.

First we observe that if $X_i\cong\MbbP^2$, then $\gamma(i)=9$ and thus $S_4(i)$
for all $0\leq i\leq \ell$.
So we may assume \Mwlog that $S_4(0)$.
Suppose that $X_i\cong\MbbP^2$ for $0\leq i\leq \ell$ \Mst $v(i+1)=v(i)+3$ at each step.
In this case it follows from \PRP{neron} that $v(\ell)\leq 3$ and $\beta(\ell)\leq-3$.
It follows from (H) and (Z) that
\[
\ell\leq \frac{h(0)-2}{6} \quad\text{and thus}\quad v\leq 3\left(\frac{h(0)-2}{6}\right)+3.
\]

For the upper bound $\tilde{\ell}$
we may assume \Mwlog that $\beta(0)<0$ since if $S_1(i)$, then $X_i\ncong\MbbP^2$
for all $0\leq i\leq \ell$.
From \THM{bound} we used \LEM{S3} and thus assumed $S_3(i)$ with $\beta(i)=\beta(\ell)=-1$
for $0\leq i \leq \ell-1$.
In this case it follows from \PRP{neron} that $v(\ell)\leq 2$.
It follows that
\[
\ell\leq \tilde{\ell}=\frac{h(0)-2}{2} \quad\text{and thus}\quad v\leq 2\left(\frac{h(0)-2}{2}\right)+2.
\]
Thus indeed we established that $3\ell+3< 2\tilde{\ell}+2$.
We conclude this proof by recalling from \SEC{family} that if $p\leq -1$, then $v(\ell)\leq 1$.
\end{proof}

\begin{remark}
\textbf{(computing invariants)}
\\
Note that $\alpha(0)$ is the degree of the (projection of the) polarized model of $(X_0,D_0)$.
From the adjunction formula it follows that the geometric genus of a generic hyperplane section of $(X_0,D_0)$ is equal to
the arithmetic genus
\[
p_a(D_0)=\frac{\alpha(0)+\beta(0)}{2}+1.
\]
It follows that $\alpha(0)$ and $\beta(0)$ can be
computed from the degree and geometric genus of a generic hyperplane section.
Let $Y\subset\MbbP^n$ be the polarized model of $(X_0,D_0)$,
so that $Y$ is linearly normal.
From \PRP{chain} it follows that
\[
n+1=h^0(D_0)=\frac{\alpha(0)-\beta(0)}{2}+p+1
\]
and thus we can compute the arithmetic genus $p$ of $X_0$.
\Mend
\end{remark}

\begin{example}
\textbf{(adjoint chains)}
\label{exm:bound}
\\
In the following four tables we represent the invariants 
from \SEC{bound} that follow the combinatorics
of adjoint chains. We denote the upper bound of \THM{bound}
by $\tilde{\ell}(i)$. See the beginning of \SEC{bound} for the remaining notation.
The heading of each table denotes the arithmetic
genus of $X_0$ and the number of different adjoint states that are reached.
The transition between adjoint states is indicated by a vertical double line.
These examples confirm that the upper bounds in \THM{bound} and \COR{family} are tight
\Mwrt the combinatorics.
The tables were constructed using \ALG{construct} (see forward).

In Table 1 the minimal pair is a weak Del Pezzo pair of degree 1.
The upper bound for the level is tight for this example and it follows the
analysis of the proof of \THM{bound}.
The polarized model of this surface is of degree $8$.
From \SEC{family} it follows that $v(X_0)=18$.

In Table 2 the minimal pair is a weak Del Pezzo pair of degree 3.
We see that the upper bound for the level is not tight in this example.
All the adjoint states are reached in this example.
If the arithmetic genus is zero, then
the upper bound  is tight if adjoint state $S_2$ is not attained,
as was the case in Table 1.

In Table 3 the minimal pair is a geometrically ruled surface
\Mst $p=-1$ and $2D+K=kF$ as in \PRP{neron}.
We find that the upper bound for the level in \THM{bound} is tight.
The upper bound for the minimal family degree in \COR{family} is also tight: $v(X_0)=17$.

In Table 4 the minimal pair is a geometrically ruled surface \Mst $p=-2$ and $D=kF$ as in \PRP{neron}.
We find that the upper bound for the level is tight.
From \COR{family} it follows that $v(X_0)\leq 19$.
From \SEC{family} and $\alpha(\ell)=0$ it follows that $v(X_0)=18$.

\begin{center}
\footnotesize
\textbf{\normalsize Table 1 (arithmetic genus 0 and 3 adjoint states)}
\\
\begin{tabular}{|l|||r|r|r||r|r|r|r|r||r|}
\hline
$i$            & $ 0$ & $ 1$ & $ 2$ & $ 3$ & $ 4$ & $ 5$ & $ 6$ & $ 7$ & $ 8$ \\\hline
\hline
$n(i)$         & $ 0$ & $ 0$ & $ 1$ & $ 0$ & $ 0$ & $ 0$ & $ 0$ & $ 1$ & $  $ \\\hline
$\gamma(i)$    & $-1$ & $-1$ & $-1$ & $ 0$ & $ 0$ & $ 0$ & $ 0$ & $ 0$ & $ 1$ \\\hline
$\beta(i)$     & $ 2$ & $ 1$ & $ 0$ & $-1$ & $-1$ & $-1$ & $-1$ & $-1$ & $-1$ \\\hline
$h(i)$         & $ 6$ & $10$ & $12$ & $12$ & $10$ & $ 8$ & $ 6$ & $ 4$ & $ 2$ \\\hline
$\alpha(i)$    & $ 8$ & $11$ & $12$ & $11$ & $ 9$ & $ 7$ & $ 5$ & $ 3$ & $ 1$ \\\hline
$\tilde{\ell}(i)$ & $ 8$ & $ 7$ & $ 6$ & $ 5$ & $ 4$ & $ 3$ & $ 2$ & $ 1$ & $ 0$ \\\hline
\end{tabular}
\\[10mm]
\textbf{\normalsize Table 2 (arithmetic genus 0 and 4 adjoint states)}
\\
\begin{tabular}{|l|||r|r|r|r|r|r||r|r||r|r|r|r||r|}
\hline
$i$            & $ 0$ & $ 1$ & $ 2$ & $ 3$ & $ 4$ & $ 5$ & $ 6$ & $ 7$ & $ 8$ & $ 9$ & $10$ & $11$ & $12$ \\\hline
\hline
$n(i)$         & $ 0$ & $ 0$ & $ 0$ & $ 0$ & $ 0$ & $ 0$ & $ 0$ & $ 1$ & $ 0$ & $ 0$ & $ 0$ & $ 3$ & $  $ \\\hline
$\gamma(i)$    & $-1$ & $-1$ & $-1$ & $-1$ & $-1$ & $-1$ & $-1$ & $-1$ & $ 0$ & $ 0$ & $ 0$ & $ 0$ & $ 3$ \\\hline
$\beta(i)$     & $ 5$ & $ 4$ & $ 3$ & $ 2$ & $ 1$ & $ 0$ & $-1$ & $-2$ & $-3$ & $-3$ & $-3$ & $-3$ & $-3$ \\\hline
$h(i)$         & $ 6$ & $16$ & $24$ & $30$ & $34$ & $36$ & $36$ & $34$ & $30$ & $24$ & $18$ & $12$ & $ 6$ \\\hline
$\alpha(i)$    & $11$ & $20$ & $27$ & $32$ & $35$ & $36$ & $35$ & $32$ & $27$ & $21$ & $15$ & $ 9$ & $ 3$ \\\hline
$\tilde{\ell}(i)$ & $23$ & $22$ & $21$ & $20$ & $19$ & $18$ & $17$ & $ 8$ & $ 4$ & $ 3$ & $ 2$ & $ 1$ & $ 0$ \\\hline
\end{tabular}
\\[10mm]
\textbf{\normalsize Table 3 (arithmetic genus -1 and 2 adjoint states)}
\\
\begin{tabular}{|l|||r|r|r||r|r|r|r|r|r|r|r|}
\hline
$i$             & $ 0$ & $ 1$ & $ 2$ & $ 3$ & $ 4$ & $ 5$ & $ 6$ & $ 7$ & $ 8$ \\\hline \hline
$n(i)$          & $ 0$ & $ 0$ & $ 1$ & $ 0$ & $ 0$ & $ 0$ & $ 0$ & $ 0$ & $  $ \\\hline
$\gamma(i)$     & $-1$ & $-1$ & $-1$ & $ 0$ & $ 0$ & $ 0$ & $ 0$ & $ 0$ & $ 0$ \\\hline
$\beta(i)$      & $ 2$ & $ 1$ & $ 0$ & $-1$ & $-1$ & $-1$ & $-1$ & $-1$ & $-1$ \\\hline
$h(i)$          & $ 6$ & $10$ & $12$ & $12$ & $10$ & $ 8$ & $ 6$ & $ 4$ & $ 2$ \\\hline
$\alpha(i)$     & $ 8$ & $11$ & $12$ & $11$ & $ 9$ & $ 7$ & $ 5$ & $ 3$ & $ 1$ \\\hline
$\tilde{\ell}(i)$  & $ 8$ & $ 7$ & $ 6$ & $ 5$ & $ 4$ & $ 3$ & $ 2$ & $ 1$ & $ 0$ \\\hline
\end{tabular}
\\[10mm]
\textbf{\normalsize Table 4 (arithmetic genus -2 and 2 adjoint states)}
\\
\begin{tabular}{|l|||r|r|r|r|r||r|r|r|r|r|r|}
\hline
$i$             & $ 0$ & $ 1$ & $ 2$ & $ 3$ & $ 4$ & $ 5$ & $ 6$ & $ 7$ & $ 8$ & $ 9$ \\\hline\hline
$n(i)$          & $ 0$ & $ 0$ & $ 0$ & $ 0$ & $ 0$ & $ 0$ & $ 0$ & $ 0$ & $ 0$ & $  $ \\\hline
$\gamma(i)$     & $-8$ & $-8$ & $-8$ & $-8$ & $-8$ & $-8$ & $-8$ & $-8$ & $-8$ & $-8$ \\\hline
$\beta(i)$      & $32$ & $24$ & $16$ & $ 8$ & $ 0$ & $-8$ & $-16$ & $-24$ & $-32$ & $-40$ \\\hline
$h(i)$          & $40$ & $104$ & $152$ & $184$ & $200$ & $200$ & $184$ & $152$ & $104$ & $40$ \\\hline
$\alpha(i)$     & $72$ & $128$ & $168$ & $192$ & $200$ & $192$ & $168$ & $128$ & $72$ & $ 0$ \\\hline
$\tilde{\ell}(i)$  & $ 9$ & $ 8$ & $ 7$ & $ 6$ & $ 5$ & $ 4$ & $ 3$ & $ 2$ & $ 1$ & $ 0$ \\\hline
\end{tabular}
\\
.\Mend
\end{center}
\end{example}

\section{Algorithm for constructing examples}

\subsection{}

We continue to use the same notation as in the previous section.
The following lemma expresses the invariants of the first ruled pair
in an adjoint chain in terms of $n(i)$ and
the invariants of the minimal pair.

\begin{lemma}
\textbf{(formulas for intersection products)}
\label{lem:formulas}
\begin{itemize}[topsep=0pt]
\Mmclaim{a}
$\alpha(0) =  \gamma(\ell) \ell^2 - 2\beta(\ell) \ell + \alpha(\ell) - \Msum{i=0}{\ell-1} (i+1)^2 n(i)$.

\Mmclaim{b}
$\beta(0)  = -\gamma(\ell) \ell + \beta(\ell) + \Msum{i=0}{\ell-1} (i+1) n(i)$.

\Mmclaim{c}
$\gamma(0) =  \gamma(\ell) - \Msum{i=0}{\ell-1} n(i)$.
\end{itemize}

\end{lemma}

\begin{proof}
Let $R_i$ be sum of exceptional curves that are contracted by $\mu_i$ for $0 \leq i < \ell$.
From the pullback formula for the canonical class in \SEC{theory}
and from \LEM{D}.\Mrefmclaim{a} it follows that
$K_{i-1}=\mu_{i-1}^*K_i+R_{i-1}$
and
$D_{i-1}=\mu_{i-1}^*D_i-K_{i-1}$.
By abuse of notation we will denote $\mu_{i-1}^*D_i$ as $D_i$ and $\mu_{i-1}^*K_i$ as $K_i$.

It follows that $D_{\ell-1} = D_\ell - K_\ell - R_{\ell-1}$.
In the next iteration we obtain
$
D_{\ell-2}
=
D_{\ell-1} - K_{\ell-1} - R_{\ell-2}
=
(D_\ell - K_\ell - R_{\ell-1}) - (K_\ell+R_{\ell-1}) - (R_{\ell-2})
=
D_\ell -2K_\ell -2R_{\ell-1} - R_{\ell-1}
$.
Repeating this we obtain
\[
D_0 = D_\ell - \ell K_\ell - \Msum{i=0}{\ell-1} (i+1) R_i.
\]
Similary we find
\[
K_0 = K_\ell + \Msum{i=0}{\ell-1} R_i.
\]
From \LEM{step}.\Mrefmclaim{c}, it follows that $R_i^2=-n(i)$.
From the projection formula in \SEC{theory}
it follows that $R_i\cdot R_j=D_\ell \cdot R_i=K_\ell \cdot R_i=0$ for $0\leq i,j \leq\ell$.
\end{proof}

The following algorithm outputs\Md
for a given level, invariants of minimal ruled pair
and the degree of the first ruled pair\Md
invariants that follow the combinatorics of an adjoint chain.
Moreover, 
the input level
is as close as possible to 
the upper bound of \THM{bound} in terms of the invariants.
The adjoint chain invariants in \EXM{bound} were constructed
with this algorithm and proof that the upper bounds
of \THM{bound} are tight in a combinatorial sense.

\newpage
\begin{algorithm}
\textbf{(construct adjoint chain)}
\label{alg:construct}
\begin{itemize}[itemsep=1pt,topsep=0pt, leftmargin=0.5cm]

\item[]\textbf{input:}
Level $\ell$, $\alpha(\ell)$, $\beta(\ell)$, $\gamma(\ell)$ and $c\in\MbbZ_{\geq1}$.

\item[]\textbf{output:}
The number of contracted curves $n(i)$ for $0 \leq i \leq \ell-1$ \Mst
the difference between $\ell$ and the upper bound in \THM{bound} is minimal,
under the condition that $\alpha(0)=c$.
If the output is $\emptyset$, then
no such valid adjoint chain exists for given input.

\item[]\textbf{method:}
Below is the description of the algorithm in pseudo code
using python syntax (\verb+#+ is for commenting).
The values \verb+l,al,bl,gl,c,None+ denote $\ell,\alpha(\ell)$, $\beta(\ell)$, $\gamma(\ell)$, $c$, $\emptyset$ respectively.
The function \\
\verb+a0(l, al, bl, gl, n)+ computes $\alpha(0)$
with the formula of \LEM{formulas}.\Mrefmclaim{a}.

\begin{Verbatim}[baselinestretch=0.9]
def construct_adjoint_chain( l, al, bl, gl, c ):

    n = l * [0]  # n is a list of l zeros [0,...,0]
    while True:

        # compute the maximal index j<=l-1 such that 
        # a0(m)>=c, where m equals n with j-th index 
        # increased by one.
        j = -1
        m = copy( n )  # m is set equal to list n
        while a0( l, al, bl, gl, m ) >= c and j <= l - 2:
            j = j + 1
            m = copy( n )
            m[j] = m[j] + 1

        if a0( l, al, bl, gl, m ) < c:
            j = j - 1

        if j >= 0:
            n[j] = n[j] + 1
        elif a0( l, al, bl, gl, n ) >= c:
            return n
        else:
            return None
\end{Verbatim}
\Mend
\end{itemize}
\end{algorithm}

%
%
%

\begin{proposition}
\textbf{(algorithm)}
\\
The output specification of \ALG{construct} is correct.

\end{proposition}

\begin{proof}
Note that the output and input
of \ALG{construct} uniquely defines a sequence
of invariants conform the rules of \LEM{step}, \LEM{state} and \LEM{dim}:
\[
(n(i),\gamma(i),\beta(i),h(i),\alpha(i))_{0\leq i \leq \ell-1}.
\]
In particular all the tables of \EXM{bound} are constructed with \ALG{construct}.
We denote by $\tilde{\ell}$ the upper bound of \THM{bound}
which depends on $\alpha(0)$, $\beta(0)$ and $p$.
From \PRP{neron} we see that if $\gamma(\ell)>0$, then $p=0$ and otherwise $\gamma(\ell)=8(p+1)$.

From \LEM{formulas}.\Mrefmclaim{a} it is immediate that the
algorithm terminates.
We prove that the algorithm outputs $n=(n(i))_i$, \Mst $\tilde{\ell}-\ell$
is minimal under the condition that $\alpha(0)=c$.

\Mclaim{1}
In order to minimize $\tilde{\ell}-\ell$
we need to minimize $\gamma(\ell)-\gamma(0)$ and maximize $\beta(0)$.

For the upper bound for $\ell$ as asserted in \LEM{S4} ($S_4(0)$)
we assumed that $\gamma(0)=\ldots=\gamma(\ell)=1$ and $\beta(0)>0$.
For the upper bound for $\ell$ as asserted in \LEM{S3} ($S_3(0)$)
we assumed that $\gamma(0)=\ldots=\gamma(\ell-1)=0$ and $\beta(0)=\ldots=\beta(\ell)=-1$.
For the upper bound for $\ell$ as asserted in
\LEM{S2a} ($S_2(0)$ and $p\geq -1$)
we assumed that $\gamma(0)=\ldots=\gamma(s-1)=-1$, $\gamma(s)=\ldots=\gamma(\ell-1)=0$
and $\beta(s)=\ldots=\beta(\ell)=-1$ \Mst $S_2(s-1)$ and $S_3(s)$.
For the upper bound for $\ell$ as asserted in
\LEM{S2b} ($S_2(0)$ and $p\leq -2$)
we assumed that $\gamma(0)=\ldots=\gamma(\ell)$.
It follows\Md under the constraints of \LEM{state}\Md
that in order to minimize $\tilde{\ell}-\ell$ we want to minimize
$\gamma(\ell)-\gamma(0)$ and maximize $\beta(0)$.
This completes the proof of \Mrefclaim{1}.

From \LEM{formulas}.\Mrefmclaim{c} we find that
we minimize $\gamma(\ell)-\gamma(0)$ if we minimize:
\[
\Gamma:=\Msum{i=0}{\ell-1}n(i).
\]

From \LEM{formulas}.\Mrefmclaim{b} we find that
we maximize $\beta(0)$ if we maximize:
\[
\Theta:=\Msum{i=0}{\ell-1}(i+1)n(i).
\]

From \LEM{formulas}.\Mrefmclaim{a} we find that
$\alpha(0) = C - \Lambda$ where
\[
\Lambda := \Msum{i=0}{\ell-1}(i+1)^2n(i),
\]
and $C$ is a constant which depends on the input.
If $C<c$, then the algorithm returns $\emptyset$.
Otherwise, we ensure that $\alpha(0)=c$
with the $n(0)$ term in $\Lambda$.

At each step of the while-loop the algorithm increases $n(i)$
with one, for as large possible $i$, under the constraint that $\alpha(0)\geq c$.
This way $\Theta$ is maximized since the coefficient of $n(i)$ is $i+1$.
The term $\Lambda$ is maximized even more since the coefficient of $n(i)$ equals $(i+1)^2$.
Therefore the condition $\alpha(0)=c$ is met in a minimal number of steps.
Thus $\Gamma$ is minimized under the constraint that $\alpha(0)\geq c$.
Now it follows from \Mrefclaim{1} that 
the output specification of \ALG{construct} is correct.
\end{proof}

\subsection{Geometric meaning of the constant c}
\label{sec:c}

Suppose that the adjoint chain of $(X_0,D_0)$
has invariants conform input and output of \ALG{construct}.
From \PRP{neron} it follows that the arithmetic genus of $X_0$ equals
$p=\min(0, \Mceil{\frac{1}{8}\gamma(\ell)-1})$.

The input constant $c$ equals the
degree of the polarized model of $(X_0,D_0)$.

We will now argue that the constant $c$ also measures the embedding dimension
of the polarized model of $(X_0,D_0)$.
Recall that the embedding dimension of the polarized model of $(X_0,D_0)$ equals $h^0(D_0)-1$.
If $D_0$ is ample, then it follows from Kodaira vanishing theorem and Riemann Roch theorem that
\[
h^0(D_0)= \frac{h(0)}{2}+p+1.
\]
By increasing the input constant $c$ we increase $h(0)$ and consequently $h^0(D_0)$.
If $D_0-K_0$ is only nef and big, then alternatively we can use the
Kawamata-Viehweg vanishing theorem.

Recall that by definition of nef and big only a high enough multiple of $D_1$
defines a birational morphism.
Reider's theorem \citep[Theorem~11.4]{bar1} says that if $c=D_0^2\geq 10$ and there exists no curve $C$ such that
($D_0\cdot C=0$ and $C^2=-1$) or ($D_0\cdot C=1$ and $C^2=0$) or
($D_0\cdot C=2$ and $C^2=0$), then $D_0+K_0$ defines a birational morphism.
Notice that $D_1$ is the pushforward of $D_0+K_0$ by definition
and thus if  $D_0+K_0$ defines a birational morphism, then 
the polarized model of $(X_1,D_1)$ is a surface.

\subsection{Computing examples from output of algorithm}

Let input $\ell$, $\alpha(\ell)$, $\beta(\ell)$, $\gamma(\ell)$, $c$ and
output $(n(i))_{0\leq i \leq \ell-1}$ of \ALG{construct} be given.
The output of
\ALG{construct} is not necessarily \Mdef{geometric}
in the sense that
$(X_0,D_0)$ exists \Mst
the polarized model of $(X_0,D_0)$ is a surface.
In particular, $\alpha(\ell)$, $\beta(\ell)$, $\gamma(\ell)$ has to be conform \PRP{neron}.
However, if the output is geometric, then we
can compute\Md at least in theory\Md equations for a polarized model of $(X_0,D_0)$
\Mst its adjoint chain has the corresponding invariants.

For the sake of simplicity we assume that $X_\ell$
is the blowup of the projective plane with $D_\ell=-K_\ell$ as in \PRP{neron}.
We blow up $X_\ell$ in $n(\ell-1)$ generic points.
It follows from \LEM{D}.\Mrefmclaim{a} and
the pullback formula for the canonical class in \SEC{theory}
that
\[
D_{\ell-1}=\mu^*D_\ell-\mu^*K_\ell+\Msum{i=1}{n(\ell-1)}E_i',
\]
where $E_i'$ are disjoint exceptional curves.
Similar as in the proof of \LEM{formulas}, we find after
sequentially taking the pullback as above that
\[
D_0=dH-\Msum{j}{}m_jE_j,
\]
where $H$ is the pullback of lines in the projective plane, $m_j>0$
and the $E_j$ are the pullback of exceptional curves.
Note that $D_0^2=c$ by assumption.

We construct a linear series $|D_0|$ in the plane with polynomials of degree $d$ and
generic base points with multiplicities $(m_i)_i$.
We check whether the map associated to the linear series parametrizes a surface,
otherwise we have to consider a multiple of $D_0$ (see \SEC{c}).
After a generic projection we may assume that we have a
parametrization of a hypersurface in 3-space.
We consider an implicit equation of degree $\alpha(0)$ with undetermined coefficients
and substitute the parametrization.
We obtain an implicit equation by solving the linear system of
equations in the undetermined coefficients.

See \citep[Example~52]{nls1} for worked out equations
for a surface of degree 8 with a minimal family of degree 8.
As illustrated in Table 1 of \EXM{bound},
a surface of degree 8
has minimal family degree of at most 18.

\section{Inequality for lattice polygons}

Let $(X_0,D_0)$ be a toric surface with polarized model $Y_0\subset \MbbP^n$.
We define the lattice polygon $P_0$ by taking the convex hull of the lattice points in the lattice $\MbbZ^2\subset\MbbR^2$
with coordinates defined by the exponents of a monomial parametrization $\Marrow{(\MbbC^*)^2}{}{Y_0}$.

We denote $\rho(0)$ for the Picard number of $X_0$.
We define $S(0)$ to be the number of exceptional divisors in the minimal resolution of
the isolated singularities of $Y_0$. We introduce the following notation:
\[
v(0):=\rho(0)+2-S(0).
\]
The \MdefAttr{adjoint}{lattice polygon} of a lattice polygon is
defined as the convex hull of its interior lattice points.
We call a lattice polygon \MdefAttr{minimal}{lattice polygon}
if its adjoint is either the empty set, a point or a line segment.
The \MdefAttr{level}{lattice polygon} $\ell(P_0)$ of a lattice polygon is defined
as the number of subsequent adjoint lattice polygons
$\Marrow{P_0}{}{}\ldots\Marrow{}{}{P_{\ell(P_0)}}$
until a minimal lattice polygon $P_{\ell(P_0)}$ is obtained.
See \RMK{levelkeel} concerning the alternative definition
for level as in \citep[Section~2.3]{sch3} and \citep[Section~3]{sch4}.

We recall part of the dictionary in \citep[Section~1]{sch5}
using the notation at the beginning of \SEC{bound}:
\begin{itemize}[topsep=0pt]
\item $\frac{\alpha(0)}{2}=a(P_0)$ where $a(P_0)$ is the area of~$P_0$,
\item $-\beta(0)=b(P_0)$ where $b(P_0)$ is the number of boundary lattice points, 
\item $v(0)=v(P_0)$ where $v(P_0)$ is the number of vertices of~$P_0$, and
\item $\ell=\ell(P_0)$ where $\ell(P_0)$ is level of~$P_0$.
\end{itemize}

From \LEM{formulas} it follows that
\begin{equation}
\label{eqn:1}
\alpha(0)+2\ell\beta(0)=-\gamma(\ell)\ell^2+\alpha(\ell) + \Phi,
\end{equation}
where
\[
\Phi:=\Msum{i=0}{\ell-1}(2\ell-i-1)(i+1) n(i).
\]
As an immediate consequence we obtain the following inequality
\begin{equation}
\label{eqn:2}
\alpha(0)+2\ell\beta(0)+\gamma(\ell)\ell^2\geq 0.
\end{equation}
From \PRP{neron} it follows that $\gamma(\ell)\leq 9$
and by substituting $9$ for $\gamma(\ell)$ in \EQN{2}
we recover the inequality of \citep[Theorem~5]{sch5}.
Moreover, we see that the inequality holds more generally for birationally ruled surfaces.
Note that for irrational birationally ruled surfaces we have that $\gamma(\ell)\leq 0$.

We want improve \EQN{2} by bounding $\Phi$ in terms of $v(0)$.
From \PRP{neron} it follows that $\rho(\ell)\leq 9$.
We recall from \SEC{theory} that the Picard number decreases by 1
for each contracted exceptional curve, and thus
\[
\Msum{i=0}{\ell-1}n(i)\geq \rho(0)-9 \geq v(0)-11.
\]
From $(2\ell-i-1)(i+1)\geq 2\ell-1$ for all $0 \leq i \leq \ell-1$
it follows that
\begin{equation}
\label{eqn:3}
\Phi\geq  (2\ell-1)\Msum{i=0}{\ell-1}n(i) \geq (2\ell-1)(v(0)-11).
\end{equation}
Now from \EQN{1}, \EQN{3} and $\alpha(\ell)\geq 0$
we obtain the following inequality on invariants of
birationally ruled surfaces:

\begin{theorem}
\[
\alpha(0)+2\ell\beta(0)+9\ell^2 \geq (2\ell-1)(v(0)-11).
\]
\end{theorem}

Restricting to toric surfaces and applying the dictionary
we obtain an improved inequality for lattice polygons, as
was conjectured in \citep[Section~4]{sch5},     \citep[Section~2.7]{cst1}:

\begin{corollary}
\[
2a(P_0)-2\ell(P_0)b(P_0)+9\ell(P_0)^2 \geq (2\ell(P_0)-1)(v(P_0)-11).
\]
\end{corollary}

\section{Acknowledgements}

I would like to thank Josef Schicho for useful discussions.
This work was supported by base funding of the
King Abdullah University of Science and Technology.

\bibliography{geometry}

\paragraph{address of author:}
Johann Radon Institute for Computational and Applied 
Mathematics (RICAM), Austrian Academy of Sciences
\\
\textbf{email:} niels.lubbes@gmail.com


\end{document}

%% file: command.tex
    \newcommand{\Mend}{\hfill \ensuremath{\vartriangleleft}}
    \newcommand{\Md}{\textemdash}

    \newcommand{\Msum}[2]{\ensuremath{\overset{#2}{\underset{#1}{\sum}}}}

    \newcommand{\Mbigfloor}[1]{{\left\lfloor #1 \right\rfloor}}
    \newcommand{\Mceil}[1]{{\lceil #1 \rceil}}

    \newcommand{\Mmod}[1]{\langle #1\rangle} 

    \newcommand{\Mset}[2]{\ensuremath{\{ #1 | #2 \}}}
    
    \newcommand{\Mdashfun}[5]{\ensuremath{#1: #2 \dashrightarrow #3,\quad #4 \mapsto #5 }}

    \newcommand{\Mand}{{\textrm{~and~}}}

    \newcommand{\Miff}{if and only if }

    \newcommand{\Mst}{such that }
    \newcommand{\Mwlog}{without loss of generality }
    
    \newcommand{\Mwrt}{with respect to }

    \newcommand{\Marrow}[3]{\ensuremath{#1\stackrel{#2}{\longrightarrow}#3}}

    \newcommand{\Mrow}[3]{\ensuremath{#1\colon #2\longrightarrow #3}}
    \newcommand{\Mdashrow}[3]{\ensuremath{#1\colon #2\dashrightarrow #3}}    
    
    \newcommand{\Mdef}[1]{\textit{#1}\index{#1}}
    \newcommand{\MdefAttr}[2]{\textit{#1}\index{#2!#1}\index{#1}}

    \newcommand{\McalO}{{\mathcal{O}}}
    \newcommand{\MbbC}{{\mathbb{C}}}\newcommand{\MbbP}{{\mathbb{P}}}\newcommand{\MbbQ}{{\mathbb{Q}}}\newcommand{\MbbR}{{\mathbb{R}}}\newcommand{\MbbZ}{{\mathbb{Z}}}

    \newcommand{\LEM}[1]{Lemma~\ref{lem:#1}}
    \newcommand{\PRP}[1]{Proposition~\ref{prp:#1}}
    \newcommand{\COR}[1]{{Corollary~\ref{cor:#1}}}
    \newcommand{\THM}[1]{{Theorem~\ref{thm:#1}}}
    \newcommand{\DEF}[1]{{Definition~\ref{def:#1}}}
    \newcommand{\RMK}[1]{{Remark~\ref{rmk:#1}}}
    \newcommand{\EQN}[1]{{(\ref{eqn:#1})}}
    \newcommand{\SEC}[1]{{\textsection\ref{sec:#1}}}
    
    \newcommand{\ALG}[1]{{Algorithm~\ref{alg:#1}}}
    
    \newcommand{\EXM}[1]{{Example~\ref{exm:#1}}}

    \newcommand{\Mmclaim}[1]{\item[\textbf{#1)}]}
    \newcommand{\Mrefmclaim}[1]{#1)}
    \newcommand{\Mclaim}[1]{\textit{Claim #1:}}
    \newcommand{\Mrefclaim}[1]{claim #1}